\documentclass[12pt]{article}
\usepackage{amssymb,amsmath,amsthm, amsfonts}
\usepackage{graphicx}
\usepackage{epsfig}
\usepackage{tikz}

\def\disp{\displaystyle}

\def\dref#1{(\ref{#1})}

\theoremstyle{plain}
\newtheorem{theorem}{Theorem}[section]
\newtheorem{lemma}{Lemma}[section]

\theoremstyle{definition}

\newtheorem{remark}{Remark}[section]

\numberwithin{equation}{section}

\begin{document}

\title{\bf Global existence and boundedness of solution of a   parabolic--parabolic--ODE chemotaxis--haptotaxis model with
(generalized) logistic
source}

\author{Ling Liu $^{1}$ \thanks{E-mail address:
 liuling2004@sohu.com (L. Liu)}£¬
Jiashan Zheng$^{a,b}$
\thanks{Corresponding author.   E-mail address:
 zhengjiashan2008@163.com (J. Zheng)}
\\
 $^{1}$
    Department of Basic Science,\\
    Jilin Jianzhu University, Changchun 130118, P.R.China \\
 $^{a}$
    School of Information,\\
    Renmin University of China, Beijing, 100872, P.R.China \\
    $^{b}$
    School of Mathematics and Statistics Science,\\
     Ludong University, Yantai 264025,  P.R.China \\
}
\date{}

\maketitle \vspace{0.3cm}
\noindent
\begin{abstract}
In this paper, we study the following chemotaxis--haptotaxis system with (generalized) logistic
source
$$
 \left\{\begin{array}{ll}
  u_t=\Delta u-\chi\nabla\cdot(u\nabla v)-
  \xi\nabla\cdot(u\nabla w)+u(a-\mu  u^{r-1}-w),\\
 \displaystyle{v_t=\Delta v- v +u},\quad
\\
\displaystyle{w_t=- vw},\quad\\
\displaystyle{\frac{\partial u}{\partial \nu}=\frac{\partial v}{\partial \nu}=\frac{\partial w}{\partial \nu}=0},\quad
x\in \partial\Omega, t>0,\\
\displaystyle{u(x,0)=u_0(x)},v(x,0)=v_0(x),w(x,0)=w_0(x),\quad
x\in \Omega,\\
 \end{array}\right.\eqno(0.1)
$$
in a smooth bounded domain  $\mathbb{R}^N(N\geq1)$, with parameter $r>1$.
 the parameters $a\in \mathbb{R}, \mu>0, \chi>0$.
 It is shown that when
$r>2$,
or
\begin{equation*}
 \mu>\mu^{*}=\begin{array}{ll}
\frac{(N-2)_{+}}{N}(\chi+C_{\beta}) C^{\frac{1}{\frac{N}{2}+1}}_{\frac{N}{2}+1},~~~\mbox{if}~~r=2,\\
 \end{array}
\end{equation*}
the considered problem possesses a global
classical solution which is bounded,
where $C^{\frac{1}{\frac{N}{2}+1}}_{\frac{N}{2}+1}$
is a positive constant which is corresponding to the maximal sobolev
regularity. Here $C_{\beta}$ is a  positive constant which depends on $\xi$, $\|u_0\|_{C(\bar{\Omega})},\|v_0\|_{W^{1,\infty}(\Omega)}$
 and $\|w_0\|_{L^\infty(\Omega)}$. This result improves or extends
previous results of several authors.
\end{abstract}

\vspace{0.3cm}
\noindent {\bf\em Key words:}~Boundedness;
Chemotaxis--haptotaxis;
(Generalized) logistic source

\noindent {\bf\em 2010 Mathematics Subject Classification}:~  92C17, 35K55,
35K59, 35K20

\newpage
\section{Introduction}
Chemotaxis is the oriented cell movement along concentration gradients of a chemical signal
produced by the cells themselves. In 1970s, a well-known chemotaxis model was proposed by Keller and Segel (\cite{Keller79}), which describes the aggregation processes of the cellular slime mold Dictyostelium discoideum.
Since then, a number of variations of the Keller--Segel model  have attracted the attention of many mathematicians, and
the focused issue was the boundedness or blow-up of the solutions (\cite{Cie791,Hillen79,Horstmann2710,Horstmann791,Winkler793,Painter79}).
The striking feature of Keller--Segel models  is the possibility
of blow-up of solutions in a finite (or infinite) time (see, e.g., \cite{Bellomo1216,Horstmann2710,Nanjundiahffger79312,Winkler793}), which
strongly depends on the space dimension. We also refer the
reader to Winkler \cite{Winkler715,Winklersdddhjjj715,Winklersddd715} (and the references therein) for some other works on the
finite-time blow up of solutions of the variants of Keller--Segel models.
Moreover, some recent studies have shown that
the blow-up of solutions can be inhibited by the nonlinear diffusion (see   {Ishida} et al. \cite{Ishida}  Winkler et al. \cite{Bellomo1216,Tao794,Winkler79,Winkler72})
and the (generalized) logistic damping (see  Li and  Xiang\cite{LiLittffggsssdddssddxxss}, Tello and  Winkler \cite{Tello710}, Wang et al. \cite{Wang76}, Zheng et al.  \cite{Zhengzseeddd0}).

In order to describe the
cancer invasion mechanism, in 2005, Chaplain and Lolas (\cite{Chaplain1}) extended the classical Keller--Segel model
where, in addition to random diffusion, cancer cells bias their movement towards a gradient of
a diffusible matrix-degrading enzyme (MDE) secreted by themselves, as well as a gradient of
a static tissue, referred to as extracellular matrix (ECM), by detecting matrix molecules such
as vitronectin adhered therein. The latter type of directed migration of cancer cells is usually
referred to as haptotaxis (see Chaplain and Lolas \cite{Chaplain7}).
According to the model proposed in \cite{Chaplain1,Chaplain7,Hillensxdc79}, in this paper,
we  consider the chemotaxis--haptotaxis system with (generalized) logistic
source
\begin{equation}
 \left\{\begin{array}{ll}
  u_t=\Delta u-\chi\nabla\cdot(u\nabla v)-\xi\nabla\cdot
  (u\nabla w)+ u(1-u^{r-1}-w),\quad
x\in \Omega, t>0,\\
 \displaystyle{v_t=\Delta v +u- v},\quad
x\in \Omega, t>0,\\
\displaystyle{w_t=- vw },\quad
x\in \Omega, t>0,\\
 \displaystyle{\frac{\partial u}{\partial \nu}=\frac{\partial v}{\partial \nu}=\frac{\partial w}{\partial \nu}=0},\quad
x\in \partial\Omega, t>0,\\
\displaystyle{u(x,0)=u_0(x)},v(x,0)=v_0(x),w(x,0)=w_0(x),\quad
x\in \Omega,\\
 \end{array}\right.\label{1.1}
\end{equation}
where
  $r>1,\Omega\subset \mathbb{R}^N(N\geq1)$ is a bounded domain with smooth boundary,  $\disp\frac{\partial}{\partial\nu}$ denotes the outward
normal derivative on $\partial\Omega$, the three variables $u, v$ and $w$ represent the cancer cell density, the
MDE concentration and the ECM density, respectively.
  The parameters $\chi,\xi$ and $\mu$ are
positive which 
 measure the chemotactic,  haptotactic sensitivities and the proliferation rate of the cells, respectively. 
As is pointed out by \cite{Bellomo1216} (see also Tao and Winkler \cite{Tao72}, Winkler \cite{Winkler715}, Zheng \cite{Zhenssdssdddfffgghjjkk1}), in this modeling context the cancer cells are also usually  assumed
to follow a generalized logistic growth $u(1-u^{r-1}-w)$ ($r>1$), which denotes the proliferation rate of the cells and competing for space with healthy tissue.
And the initial data $(u_0, v_0,w_0)$ supposed to be satisfied the following
conditions
\begin{equation}\label{x1.731426677gg}
\left\{
\begin{array}{ll}
\displaystyle{u_0\in C(\bar{\Omega})~~\mbox{with}~~u_0\geq0~~\mbox{in}~~\Omega~~\mbox{and}~~u_0\not\equiv0},\\
\displaystyle{v_0\in W^{1,\infty}(\Omega)~~\mbox{with}~~v_0\geq0~~\mbox{in}~~\Omega},\\
\displaystyle{w_0\in C^{2+\vartheta}(\bar{\Omega})~~\mbox{with}~~w_0\geq0~~\mbox{in}~~\bar{\Omega}~~\mbox{and}~~\frac{\partial w_0}{\partial\nu}=0~~\mbox{on}~~\partial\Omega} ~~~~~~~~~~~~~~~~~~~~~~~~~~~~~~~~~~~~~~~~~~~~~~~~~~~~~~~~~~~~~~~~~~~\\
\end{array}
\right.
\end{equation}
with some $\vartheta\in(0,1).$

In order to better understand model \dref{1.1}, let us mention 
 the following quasilinear chemotaxis--haptotaxis system, which is a closely related variant of \dref{1.1}
 \begin{equation}
 \left\{\begin{array}{ll}
  u_t=\nabla\cdot(\phi(u)\nabla u)-\chi\nabla\cdot(u\nabla v)-\xi\nabla\cdot
  (u\nabla w)+ \mu u(1-u^{r-1}-w),\quad
x\in \Omega, t>0,\\
 \displaystyle{\tau v_t=\Delta v +u- v},\quad
x\in \Omega, t>0,\\
\displaystyle{w_t=- vw },\quad
x\in \Omega, t>0,\\
 \displaystyle{\frac{\partial u}{\partial \nu}=\frac{\partial v}{\partial \nu}=\frac{\partial w}{\partial \nu}=0},\quad
x\in \partial\Omega, t>0,\\
\displaystyle{u(x,0)=u_0(x)},v(x,0)=v_0(x),w(x,0)=w_0(x),\quad
x\in \Omega,\\
 \end{array}\right.\label{ffbbggnjkk1.1}
\end{equation}
where $\mu\geq0,\tau\in\{0,1\}$,
the function $\phi(u)$
 fulfills
\begin{equation}\label{9161}
\phi\in  C^{2}([0,\infty))
\end{equation}
and there
exist  constants $m\geq1$ and  $C_{\phi}$
such that
\begin{equation}\label{9162}
\phi(u) \geq C_{\phi}(u+1)^{m-1}~ \mbox{for all}~ u\geq0.
\end{equation}

When $w \equiv 0$, \dref{ffbbggnjkk1.1} is reduced to the chemotaxis-only system  with (generalized) logistic source (see  Xiang \cite{Xiangssdd55672gg},
Zheng et al.
\cite{Zheng0,Zheng33312186,Zhengsssddsseedssddxxss,Zhengghjjkk1,Zhddddffengssdsseerrteeezseeddd0}). And
global existence, boundedness  and asymptotic behavior of solution were studied in \cite{Lianu1,Marciniak,Walker,Zhengsssddswwerrrseedssddxxss}.
Going beyond the above statements, we  should mention the papers \cite{Winkler793}  and \cite{Winklersdddhjjj715} (and references therein), which deal with the  blow-up
of solutions to parabolic-elliptic versions of \dref{ffbbggnjkk1.1} (with $w\equiv0$). For example,   when $D(u)\equiv1$ and $N\geq3$,   it is demonstrated in  \cite{Winklersdddhjjj715} that a superlinear growth condition
on logistic
source may be insufficient to prevent finite time blow-ups for a parabolic-elliptic system of
\dref{ffbbggnjkk1.1}, which is the first rigorous detection of blow-up in a superlinearly dampened of  Keller-Segel system in {\bf three}-dimensional case. 
 From a theoretical point of view, due to the fact that the chemotaxis and haptotaxis terms require different $L^p$-estimate techniques, the problem related to the chemotaxis--haptotaxis models of cancer invasion presents an important mathematical challenging.
 There are only few results on the mathematical analysis of this (quasilinear)  chemotaxis--haptotaxis system \dref{ffbbggnjkk1.1} (Cao \cite{Cao}, Zheng et al. \cite{liujinijjkk1}, Tao  et al. \cite{Taox3201,Tao3,Tao2,Tao72,Taox26,Taox26216,Tao79477},  Wang et al. \cite{liujinijjkk1,Wangscd331629,Wangrrssdeeefffjkk1,Zhengghjjkk1}).
Indeed, if MDEs diffuses much faster than cells (see \cite{Jger317,Taox26216}),
 ($\tau = 0$ in  the second equation
  of \dref{ffbbggnjkk1.1}),  \dref{ffbbggnjkk1.1} is reduced to the parabolic--ODE--elliptic chemotaxis--haptotaxis system
  the (generalized) logistic source.
  To the best of our knowledge, there exist some boundedness and stabilization results on the simplified parabolic--elliptic--ODE chemotaxis--haptotaxis model \cite{Tao2,Taox26216,Taox26}.
When,
$r=2$ in the first equation of \dref{ffbbggnjkk1.1},   
the global boundedness of solutions to the chemotaxis--haptotaxis system with the standard   logistic source
 has been proved for any $\mu>0$ in two dimensions and for large ${\mu}$ (compared to the chemotactic sensitivity $\chi$) in three dimensions (see Tao and Wang \cite{Tao2}).
In \cite{Taox26216}, Tao and Winkler studied the global boundedness for model \dref{ffbbggnjkk1.1}
 under the condition ${\bf \mu > \frac{(N-2)^+}{N}\chi}$, moreover, in additional explicit
smallness on $w_0$, they gave the exponential decay of $w$ in the large time limit.
While if $r>1$ (the (generalized) logistic source), one can see  \cite{Zhengghjjkk1}.

As for parabolic--ODE--parabolic system  \dref{ffbbggnjkk1.1}, if $r=2,$ there has been some progress made in two or three dimensions
(see Cao  \cite{Cao} Tao and Winkler \cite{Tao1,Tao3,Taox26216}). In fact, when $\phi\equiv1$, 
Cao (\cite{Cao}) and Tao (\cite{Taox3201})   proved that \dref{ffbbggnjkk1.1} admits a unique, smooth and bounded solution if
 $\mu>0$ on $N=2$ and $\mu$ is {\bf large enough} on $N=3$. Recently, assume that $\mu$ is {\bf large enough} and $3\leq N\leq 8$, 
 the boundedness of
the global solution of system  \dref{ffbbggnjkk1.1} are obtained by Wang and Ke in \cite{Wangrrssdeeefffjkk1}.
However, they  did not give the lower bound estimation for the logistic source.
Note that the global existence and boundedness of solutions to  \dref{ffbbggnjkk1.1} is still open in three dimensions for {\bf small $\mu >0$} and in higher dimensions.


%
%
%


The main  object of the present paper is to address the  boundedness to solutions of \dref{1.1} without any restriction on the {\bf space dimension}. Our main result is the following.

\begin{theorem}\label{theorem3}Assume
that 
the initial data $(u_0,v_0,w_0)$ fulfills \dref{x1.731426677gg}. For any $N\geq1$,
if one of the following cases holds:

(i)~$r>2$;
%
%

(ii)
   \begin{equation}
\mu> \mu^{*}=\begin{array}{ll}
\frac{(N-2)_{+}}{N}(\chi+C_{\beta}) C^{\frac{1}{\frac{N}{2}+1}}_{\frac{N}{2}+1},~~~\mbox{if}~~r=2,\\
 \end{array}\label{gnjjmmx1.731426677gg}
\end{equation}
%
%
%
%
then there exists a triple $(u,v,w)\in (C^0(\bar{\Omega}\times[0,\infty))\cap C^{2,1}
(\bar{\Omega}\times(0,\infty)))^3$ which solves \dref{1.1} in the classical sense, where $C^{\frac{1}{\frac{N}{2}+1}}_{\frac{N}{2}+1}$ is a positive constant which is corresponding to the maximal sobolev
regularity. Here $C_{\beta}$ are positive constants which depends on $\xi$, $\|u_0\|_{C(\bar{\Omega})},\|v_0\|_{W^{1,\infty}(\Omega)}$ and $\|w_0\|_{L^\infty(\Omega)}$.
Moreover,  $u$, $v$  and $w$ are bounded in $\Omega\times(0,\infty)$.
\end{theorem}
\begin{remark}
(i) From Theorem \ref{theorem3}, we derive that
  the global boundedness of the solution for the  complete parabolic--parabolic and parabolic--elliptic models, which
need a coefficient of the logistic source to keep the same (except a constant).

(ii) Obviously, if $r=2$ and
$\mu>\frac{(N-2)_{+}}{N}(\chi+C_{\beta}) C^{\frac{1}{\frac{N}{2}+1}}_{\frac{N}{2}+1}$, thus, Theorem \ref{theorem3}
extends   the results of  Theorem 1.1  of Cao (\cite{Cao}), who proved the boundedness in the case $N=3$,
$r=2$ and
$\mu$ is {\bf appropriately large}.

 (iii) 
Obviously, if $r=2$
and
$\mu>\frac{(N-2)_{+}}{N}(\chi+C_{\beta}) C^{\frac{1}{\frac{N}{2}+1}}_{\frac{N}{2}+1}$, hence Theorem \ref{theorem3}
extends   the results of  Theorem 1.1  of Wang and  Ke (\cite{Wangrrssdeeefffjkk1}), who proved the boundedness of the solutions in the case $3\leq N\leq8$,
$r=2$ and
$\mu$ {\bf is appropriately large}.

(iv) Obviously, if $r>2,$ then,
$2<\frac{N+2}{2}$, therefore, Theorem \ref{theorem3} (partly) extends  the results of  Theorem 1.1  of   Zheng (\cite{Zhenssdssdddfffgghjjkk1}), who showed the
 boundedness of the solutions
in
 the cases  $r>\frac{N+2}{2}$.

 (v) If $w\equiv0$, (the PDE system \dref{1.1} is reduced to the chemotaxis-only system), it
is not difficult to obtain that the solutions under the conditions of Theorem \ref{theorem3} are
uniformly bounded when $r=2$ and
$\mu>\frac{(N-2)_{+}}{N}(\chi+C_\beta) C^{\frac{1}{\frac{N}{2}+1}}_{\frac{N}{2}+1}$, which extends and coincides with the results of Winkler (see Theorem 0.1 of \cite{Winkler37103}) and  the result of  Osaki et al. (\cite{Osakix391}), respectively.

 (vi) From Theorem \ref{theorem3}, we derive that  solutions of model \dref{1.1}
are global and bounded for any $r=2,\mu>0$ and $N\leq2$, which coincides with the result of  Tao (\cite{Taox3201}).

%

 (vii) With the help of precise estimation,
 the ideas of our paper can also be used to deal with the three-dimensional
chemotaxis-fluid system with (generalized) logistic
source.
\end{remark}


If $\phi$ is a nonlinear function of $u$, then \dref{ffbbggnjkk1.1} becomes a  quasilinear parabolic--ODE--parabolic chemotaxis--haptotaxis system.
There are only few results on the mathematical analysis of this quasilinear parabolic--ODE--parabolic chemotaxis--haptotaxis system with the standard  logistic
source ($r=2$ in  the first equation of \dref{ffbbggnjkk1.1}).
In fact, if $N=2$, Zheng et al. (\cite{Zhenghhjjghjjkk1}) mainly studied the global boundedness for model \dref{ffbbggnjkk1.1}
with $\phi$ satisfies \dref{9161}--\dref{9162} and
$m>1$.
While,  Tao and Winkler (\cite{Tao72}) proved that model \dref{ffbbggnjkk1.1} possesses at least one nonnegative
{\bf global} classical solution when $\phi$ satisfy \dref{9161}--\dref{9162} with
$m>\max\{1,\bar{m}\}$
and
\begin{equation}\label{dcfvgg7101.2x19x318jkl}
\bar{m}:=\left\{\begin{array}{ll}
\frac{2N^2+4N-4}{N(N+4)}~~\mbox{if}~~
N \leq 8,\\
 \frac{2N^2+3N+2-\sqrt{8N(N+1)}}{N(N+20)}~~\mbox{if}~~
N \geq 9.\\
 \end{array}\right.
\end{equation}
Further, by using the boundedness of $\int_{\Omega}|\nabla v|^{l}(1\leq l<\frac{N}{N-1})$, assuming that $m > 2 - \frac{2}{N}$, Wang (\cite{Wangscd331629}) obtained the boundedness of the global solutions to \dref{ffbbggnjkk1.1}.
Recently, with the help of the boundness of $\int_{\Omega}|\nabla v|^{2}$, Zheng (\cite{Zhddengssdeeezseeddd0}) extends the results of \cite{Wangscd331629} when $m>\frac{2N}{N+2}$.
%
%
%
Very recently, it is asserted that
if $\phi$ satisfies \dref{9161}--\dref{9162} and
\begin{equation}\label{sdedrtggyyhy916ddffffttyy2}
 m\left\{\begin{array}{ll}
>2-\frac{2}{N}~~\mbox{if}~~
1<r<\frac{N+2}{N},\\
 >1+\frac{(N+2-2r)^+}{N+2}~~~~~~\mbox{if}~~
 \frac{N+2}{2}\geq r\geq\frac{N+2}{N},\\
 \geq1~~~~~~\mbox{if}~~ r>\frac{N+2}{2},\\
 \end{array}\right.
\end{equation}
we (\cite{Zhenssdssdddfffgghjjkk1}) proved that the unique nonnegative classical solution of
quasilinear parabolic--ODE--parabolic chemotaxis--haptotaxis system with generalized logistic
source ($r>1$ in  the first equation of \dref{ffbbggnjkk1.1}) which is global in time and bounded, however, we have to leave open here the question of how far the above hypothesis \dref{sdedrtggyyhy916ddffffttyy2}
is sharp.


It is worth to remark the main idea underlying the  proof of our results.
The proof of theorem \ref{theorem3} is based on an iterative $L^p$ estimation argument involving the  maximal Sobolev regularity and the
Moser-type limit procedure. Indeed, with the help of
 $$\min_{y>0}(y+\frac{1}{ q_0+1}\left[\frac{ q_0+1}{ q_0}\right]^{- q_0 }\left(\frac{q_0-1}{q_0} \right)^{ q_0+1}y^{- q_0 }\chi^{ q_0+1}C_{ q_0+1})=\frac{(q_0-1)}{q_0}C_{q_0+1}^{\frac{1}{q_0+1}}\chi$$
 and
 $\min_{y>0}(y+\frac{1}{ q_0+1}\left[\frac{ q_0+1}{ q_0}\right]^{- q_0 }\left(\frac{q_0-1}{q_0} \right)^{ q_0+1}y^{- q_0 }C_{\beta}^{ q_0+1}C_{ q_0+1})=\frac{(q_0-1)}{q_0}C_{q_0+1}^{\frac{1}{q_0+1}}C_{\beta}$
 (see Lemma \ref{lemma45630223116}),
%
 we can obtain  a subtle combination of entropy like
estimates for
\begin{equation}\int_{\Omega}{u^{q_0}} ~~~\mbox{for some}~~ q_0 > \frac{N}{2}.
\label{hhjkkhnjmkk7101.2x1}
\end{equation}
Then  we shall involve the variation-of-constants formula
for variable $v$ to gain
%
%
%
\begin{equation}
\int_{\Omega}|\nabla {v}|^{q}\leq C~~\mbox{for all}~~ \mbox{and}~~q\in[1,\frac{N{q_0}}{(N-{q_0})^+}),
\label{hnjmkk7101.2x1}
\end{equation}
 and thereby
 establish the  a priori estimates of the
functional
\begin{equation}
\int_{\Omega}u^{p}\leq C~~\mbox{for all}~~ \mbox{and}~~p>1.
\label{hnjmkk71ccfvvv01.2x1}
\end{equation}
Finally, in light of the Moser iteration method (see e.g.  Lemma A.1 of \cite{Tao794}) and the standard estimate for Neumann semigroup, we  established the $L^\infty$ bound of $u$ (see the proof of  Theorem \ref{theorem3}).

\section{Preliminaries and  main results}

Firstly, we recall some preliminary
lemmas, which play essential roles in our subsequent analysis.
To begin with, let us collect some basic solution properties which essentially have already been used
in \cite{Horstmann791} (see also Winkler \cite{Winkler792}, Zhang and Li \cite{Zhangddff4556}).
\begin{lemma}(\cite{Horstmann791})\label{lemmaggbb41ffgg}
For $p\in(1,\infty)$, let $A := A_p$ denote the sectorial operator defined by
\begin{equation}\label{hnjmkfgbhnn6291}
A_pu :=-\Delta u~~\mbox{for all}~~u\in D(A_p) :=\{\varphi\in W^{2,p}(\Omega)|\frac{\partial\varphi}{\partial \nu}|_{\partial\Omega}=0\}.
\end{equation}
The operator $A + 1$ possesses fractional powers $(A + 1)^{\alpha}(\alpha\geq0)$, the domains of
which have the embedding properties
\begin{equation}\label{hnjmkfgbhnn6291}
D((A+1)^\alpha)\hookrightarrow W^{1,p}(\Omega)~~\mbox{if}~~\alpha>\frac{1}{2}.
\end{equation}

If $m\in \{0, 1\}$, $p\in [1,\infty]$ and $q \in (1,\infty)$ with $m-\frac{N}{p} < 2\alpha-\frac{N}{q} $, then we have
\begin{equation}\label{hnssedrrffjmkfgbhnn6291}
\|u\|_{W^{m,p}(\Omega)}\leq C\|(A+1)^\alpha u\|_{L^{q}(\Omega)}~~~\mbox{for all}~~u\in D((A+1)^\alpha),
\end{equation}
 where $C$ is a positive constant.
The fact that the spectrum of $A$ is a $p$-independent countable set of positive real numbers
$0 = \mu_0 < \mu_1 <\mu_2 <\cdots $
entails the following consequences:
For all $1\leq p < q < \infty$ and $u\in L^p(\Omega)$ the
general $L^p$-$L^q$ estimate
\begin{equation}\label{gbhnhnjmkfgbhnn6291}
\|(A+1)^\alpha e^{-tA}u\|_{L^q(\Omega)}\leq ct^{-\alpha-\frac{N}{2}(\frac{1}{p}-\frac{1}{q})}e^{(1-\mu)t}\|u\|_{L^p(\Omega)},
\end{equation}
for any $t > 0$ and
$\alpha\geq0$ with some 	$\mu > 0.$
\end{lemma}

\begin{lemma}(\cite{Hajaiej,Ishida})\label{lemma41ffgg}
Let  $s\geq1$ and $q\geq1$.
Assume that $p >0$ and $a\in(0,1)$ satisfy
$$\frac{1}{2}-\frac{p}{N}=(1-a)\frac{q}{s}+a(\frac{1}{2}-\frac{1}{N})~~\mbox{and}~~p\leq a.$$
Then there exist $c_0, c'_0 >0$ such that for all $u\in W^{1,2}(\Omega)\cap L^{\frac{s}{q}}(\Omega)$,
$$\|u\|_{W^{p,2}(\Omega)} \leq c_{0}\|\nabla u\|_{L^{2}(\Omega)}^{a}\|u\|^{1-a}_{L^{\frac{s}{q}}(\Omega)}+c'_0\|u\|_{L^{\frac{s}{q}}(\Omega)}.$$
\end{lemma}

\begin{lemma}\label{lemma45xy1222232}(\cite{Cao,Zhddengssdsseerrteeezseedhjjjdd0})
Suppose  $\gamma\in (1,+\infty)$. Consider the following evolution equation
 \begin{equation}
 \left\{\begin{array}{ll}
v_t -\Delta v=g,~~(x, t)\in
\Omega\times(0, T ),\\
\disp\frac{\partial v}{\partial \nu}=0,~~(x, t)\in
 \partial\Omega\times(0, T ),\\
v(x,0)=v_0(x),~~~(x, t)\in
 \Omega.\\
 \end{array}\right.\label{1.3xcx29}
\end{equation}
For each $v_0\in W^{2,\gamma}(\Omega)$
such that $\disp\frac{\partial v_0}{\partial \nu}=0$ and any $g\in L^\gamma((0, T); L^\gamma(
\Omega))$, there exists a unique solution
$v\in W^{1,\gamma}((0,T);L^\gamma(\Omega))\cap L^{\gamma}((0,T);W^{2,\gamma}(\Omega)).$ Moveover, there exists a positive constant $\delta_0$ such that
\begin{equation}
\begin{array}{rl}
&\displaystyle{\int_0^T\|v(\cdot,t)\|^{\gamma}_{L^{\gamma}(\Omega)}dt+\int_0^T\|v_t(\cdot,t)\|^{\gamma}_{L^{\gamma}(\Omega)}dt+\int_0^T\|\Delta v(\cdot,t)\|^{\gamma}_{L^{\gamma}(\Omega)}dt}\\
\leq&\displaystyle{\delta_0\left(\int_0^T\|g(\cdot,t)\|^{\gamma}_{L^{\gamma}(\Omega)}dt+\|v_0(\cdot,t)\|^{\gamma}_{L^{\gamma}(\Omega)}+\|\Delta v_0(\cdot,t)\|^{\gamma}_{L^{\gamma}(\Omega)}\right).}\\
\end{array}
\label{cz2.5bbv}
\end{equation}
On the other hand, consider the following evolution equation
 \begin{equation}
 \left\{\begin{array}{ll}
v_t -\Delta v+v=g,~~~(x, t)\in
 \Omega\times(0, T ),\\
\disp\frac{\partial v}{\partial \nu}=0,~~~(x, t)\in
 \partial\Omega\times(0, T ),\\
v(x,0)=v_0(x),~~~(x, t)\in
 \Omega.\\
 \end{array}\right.\label{1.3xcx29}
\end{equation}
Then there exists a positive constant $C_{\gamma}:=C_{\gamma,|\Omega|}$ such that if $s_0\in[0,T)$, $v(\cdot,s_0)\in W^{2,\gamma}(\Omega)(\gamma>N)$ with $\disp\frac{\partial v(\cdot,s_0)}{\partial \nu}=0,$ then
\begin{equation}
\begin{array}{rl}
&\displaystyle{\int_{s_0}^Te^{\gamma s}(\| v(\cdot,t)\|^{\gamma}_{L^{\gamma}(\Omega)}+\|\Delta v(\cdot,t)\|^{\gamma}_{L^{\gamma}(\Omega)})ds}\\
\leq &\displaystyle{C_{\gamma}\left(\int_{s_0}^Te^{\gamma s}
\|g(\cdot,s)\|^{\gamma}_{L^{\gamma}(\Omega)}ds+e^{\gamma s}(\|v_0(\cdot,s_0)\|^{\gamma}_{L^{\gamma}(\Omega)}+\|\Delta v_0(\cdot,s_0)\|^{\gamma}_{L^{\gamma}(\Omega)})\right).}\\
\end{array}
\label{cz2.5bbv114}
\end{equation}
\end{lemma}

The following lemma deals with local-in-time existence and uniqueness of a classical solution for the
problem \dref{1.1} (see \cite{Tao72,liujinijjkk1}).
\begin{lemma}\label{lemma70}
Assume that the nonnegative functions $u_0,v_0,$ and $w_0$ satisfies \dref{x1.731426677gg}
for some $\vartheta\in(0,1).$
%
%
 Then there exists a maximal existence time $T_{max}\in(0,\infty]$ and a triple of  nonnegative functions 
 $$
\left\{\begin{array}{rl}
&\displaystyle{u\in C^0(\bar{\Omega}\times[0,T_{max}))\cap C^{2,1}(\bar{\Omega}\times(0,T_{max})),}\\
&\displaystyle{v\in C^0(\bar{\Omega}\times[0,T_{max}))\cap C^{2,1}(\bar{\Omega}\times(0,T_{max})),}\\
&\displaystyle{w\in  C^{2,1}(\bar{\Omega}\times[0,T_{max})),}\\
\end{array}\right.
$$
 which solves \dref{1.1}  classically and satisfies $w\leq \|w_0\|_{L^\infty(\Omega)}$
  in $\Omega\times(0,T_{max})$.
%
Moreover, if  $T_{max}<+\infty$, then
\begin{equation}
\|u(\cdot, t)\|_{L^\infty(\Omega)}+\|v(\cdot,t)\|_{W^{1,\infty}(\Omega)}\rightarrow\infty~~ \mbox{as}~~ t\nearrow T_{max}.
\label{1.163072x}
\end{equation}
\end{lemma}

Firstly, by Lemma \ref{lemma70}, we can  pick  $s_0\in(0,T_{max})$, $s_0\leq1$ and
$\beta>0$ such that
\begin{equation}\label{eqx45xx12112}
\|u(\tau)\|_{L^\infty(\Omega)}\leq \beta~~~\|v(\tau)\|_{W^{1,\infty}(\Omega)}\leq \beta~~\mbox{and}~~\|w(\tau)\|_{W^{2,\infty}(\Omega)}\leq \beta~~\mbox{for all}~~\tau\in[0,s_0].
\end{equation}

\section{The proof of theorem \ref{theorem3}}

In this section, we are going to establish an iteration step to develop the main ingredient of our result.
The iteration depends on a series of a priori estimates.
Firstly, based on the ideas of Lemma 3.1  in \cite{Wangrrssdeeefffjkk1} (see also Lemma 2.1 of \cite{Winkler37103}), we
 recall now a well-known property of systems of type \dref{1.1} with a generalized logistic
source exhibiting a  decay with respect to $u$ in the first equation.
%
%
%

\begin{lemma}\label{wsdelemma45}
Under the assumptions in theorem \ref{theorem3}, we derive that
there exists a positive constant 
$C$
such that the solution of \dref{1.1} satisfies
%
%
\begin{equation}
\int_{\Omega}{u(x,t)}+\int_{\Omega} {v^2}(x,t)+\int_{\Omega}|\nabla {v}(x,t)|^2 \leq C~~\mbox{for all}~~ t\in(0, T_{max}).
\label{cz2.5ghju48cfg924ghyuji}
\end{equation}
Moreover,
for each $T\in(0, T_{max})$, one can find a constant $C > 0$ independent of $\varepsilon$ such that
\begin{equation}
\begin{array}{rl}
&\displaystyle{\int_{0}^T\int_{\Omega}[|\nabla {v}|^2+u^r+ |\Delta {v}|^2]\leq C.}\\
\end{array}
\label{bnmbncz2.5ghhjuyuivvbssdddeennihjj}
\end{equation}
\end{lemma}

Now, applying almost exactly the same arguments as in the proof of Lemma 3.2 of \cite{Wangscd331629}
(the minor necessary changes are left as an easy exercise to the reader),
we conclude the following Lemma:
\begin{lemma}\label{zsxxcdvlemma45630}
Let $(u,v,w)$ be a solution to \dref{1.1} on $(0,T_{max})$.
  Then for any  $k>1$,
there exists a positive constant $C_\beta:=C(\xi,\|w_0\|_{L^\infty(\Omega)},\beta)$ which depends  on $\xi,\|w_0\|_{L^\infty(\Omega)}$ and $\beta$ such that
\begin{equation}
\begin{array}{rl}
&\displaystyle{-\xi\int_\Omega  u^{k-1}\nabla\cdot(u\nabla w )\leq C_\beta(\frac{k-1}{k}
\int_\Omega  u^{k}(v+1)+k\int_{\Omega} u^{k-1}|\nabla u|).}\\
\end{array}
\label{vbgncz2.5xx1ffgghh512}
\end{equation}
where $\beta$ is the same as \dref{eqx45xx12112}.
\end{lemma}
\begin{proof}
Here and throughout the proof of Lemma  \ref{zsxxcdvlemma45630}, we shall denote by $M_i (i\in N)$ several positive
constants independent of  $k$.
Firstly, observing that
the third equation of  \dref{1.1} is an ODE,  we derive that
\begin{equation}
\begin{array}{rl}
&\displaystyle{w(x,t)=w(x,s_0)e^{-\int_0^t v(x,s)ds},~~(x,t)\in\Omega\times(0,T_{max}).}\\
\end{array}
\label{vbgncz2.5xx1ffsskkoppgghh512}
\end{equation}
Hence,  by a basic
calculation, we conclude that
\begin{equation}
\begin{array}{rl}
&\displaystyle{\nabla w(x,t)=\nabla w(x,s_0)e^{-\int_0^t v(x,s)ds}}\\
&\displaystyle{-w(x,s_0)e^{-\int_0^t v(x,s)ds}\int_0^t  \nabla v(x,s)ds,~(x,t)\in\Omega\times(0,T_{max}).}\\
\end{array}
\label{vbgncz2.5xx1ffsskkopffggpgghh512}
\end{equation}
and
\begin{equation}
\begin{array}{rl}
&\Delta w(x,t)\\
\geq&\displaystyle{ \Delta w(x,s_0)e^{-\int_0^t v(x,s)ds}-2\nabla w(x,s_0)\cdot \int_0^t  \nabla v(x,s)dse^{-\int_0^t v(x,s)ds}}\\
&-\displaystyle{w(x,s_0)e^{-\int_0^t\Delta v(x,s)ds}\int_0^t\Delta v(x,s)ds.}\\
\end{array}
\label{vbgncz2.5xx1ffsskkosxdcfvppgghh512}
\end{equation}
On the other hand, for any $k\geq 1,$ integrating by parts yields
\begin{equation}
\begin{array}{rl}
&-\xi\disp\int_\Omega  u^{k-1}\nabla\cdot(u\nabla w )\\
=&\displaystyle{ -\xi\frac{k-1}{k}\int_\Omega  u^{k}\Delta w }\\
\leq&\displaystyle{ \xi\frac{k-1}{k}\int_\Omega  u^{k}(-\Delta w(x,s_0)e^{-\int_0^t v(x,s)ds}+2\nabla w(x,s_0)\cdot \int_0^t  \nabla v(x,s)dse^{-\int_0^t v(x,s)ds}) }\\
&+\displaystyle{ \xi\frac{k-1}{k}\int_\Omega  u^{k}w(x,s_0)e^{-\int_0^t\Delta v(x,s)ds}\int_0^t\Delta v(x,s)ds }\\
:=&\displaystyle{ J_1.}\\
\end{array}
\label{vbgncz2.5xxeffghjxx1ffgghh512}
\end{equation}
Now, 
using $v\geq0$ and the Young inequality,  we have
\begin{equation}
\begin{array}{rl}
J_1\leq&\displaystyle{-\xi\frac{k-1}{k}\int_\Omega  u^{k}
\Delta w(x,s_0)e^{-\int_0^t v(x,s)ds}+\xi\frac{k-1}{k}\int_\Omega  u^{k}w(x,s_0)e^{-\int_0^t v(x,s)ds}\int_0^t\Delta v(x,s)ds}\\
&+\displaystyle{\xi\frac{2(k-1)}{k}\int_\Omega  u^{k}\nabla w(x,s_0)\cdot \int_0^t  \nabla v(x,s)dse^{-\int_0^t v(x,s)ds}}\\
=&\displaystyle{-\xi\frac{k-1}{k}\int_\Omega  u^{k}
\Delta w(x,s_0)e^{-\int_0^t v(x,s)ds}+\xi\frac{k-1}{k}\int_\Omega  u^{k}w(x,s_0)e^{-\int_0^t v(x,s)ds}\int_0^t\Delta v(x,s)ds}\\
&-\displaystyle{\xi\frac{2(k-1)}{k}\int_\Omega  u^{k}\nabla w(x,s_0)\cdot\nabla e^{-\int_0^t v(x,s)ds}}\\
\leq&\displaystyle{\xi\beta\int_\Omega  u^{k}+\xi\frac{k-1}{k}\int_\Omega  u^{k}w(x,s_0)e^{-\int_0^t v(x,s)ds}\int_0^t\Delta v(x,s)ds}\\
&+\displaystyle{2\xi(k-1)\int_\Omega  u^{k-1}\nabla u\cdot\nabla w(x,s_0) e^{-\int_0^t v(x,s)ds}+\frac{2(k-1)}{k}\int_\Omega  u^{k}\Delta w(x,s_0)e^{-\int_0^t v(x,s)ds}}\\
\leq&\displaystyle{M_1\int_\Omega  u^{k}+\xi\frac{k-1}{k}\int_\Omega  u^{k}w(x,s_0)e^{-\int_0^t v(x,s)ds}\int_0^t\Delta v(x,s)ds}\\
&+\displaystyle{M_2(k\int_\Omega  u^{k-1}|\nabla u|+\int_\Omega  u^{k}),}\\
\end{array}
\label{vbgncz2.5xx1fnasddccvvfsskkosxdcfvppgghh512}
\end{equation}
where $M_1=\xi\beta$ and $M_2=\max\{2\xi\sup_{x\in\Omega}|\nabla w(x,s_0)|,2\sup_{x\in\Omega}|\Delta w(x,s_0)|\}$.
Next, due to the second equality of \dref{1.1} and $u\geq0$, we conclude that 
\begin{equation}
\begin{array}{rl}
&\displaystyle{\xi\frac{k-1}{k}\int_\Omega  u^{k}w(x,s_0)e^{-\int_0^t v(x,s)ds}\int_0^t\Delta v(x,s)ds}\\
=&\displaystyle{\xi\frac{k-1}{k}\int_\Omega  u^{k}w(x,s_0)e^{-\int_0^t v(x,s)ds}\int_0^t(v_s(x,s)+v(x,s)-u(x,s)))ds}\\
\leq&\displaystyle{\xi\frac{k-1}{k}\int_\Omega  u^{k}w(x,s_0)e^{-\int_0^t v(x,s)ds}
(v(x,t)-v_0(x)+\int_0^t v(x,s)ds}\\
\leq&\displaystyle{\xi\frac{k-1}{k}\|w_0\|_{L^\infty(\Omega)}\int_\Omega  u^{k}(v+1).}\\
\end{array}
\label{vbgncz2.5xx1fnasddccfggvvfddfghhsskkosxdcfvppgghh512}
\end{equation}
Here we have use the fact that $\frac{t}{e^t}\leq1$ (for all $t\geq0$).
Collecting \dref{vbgncz2.5xx1fnasddccvvfsskkosxdcfvppgghh512} with \dref{vbgncz2.5xx1fnasddccfggvvfddfghhsskkosxdcfvppgghh512}, we can get
the result.
\end{proof}

\begin{lemma}\label{lemma45630223116}
Let \begin{equation}
{A}_1=\frac{1}{ \delta+1}\left[\frac{ \delta+1}{ \delta}\right]^{- \delta }\left(\frac{\delta-1}{\delta} \right)^{ \delta+1},
\label{zjscz2.5297x9630222211444125}
\end{equation}
 $H(y)=y+{A}_1y^{- \delta }\chi^{ \delta+1}C_{ \delta+1}$ and $\tilde{H}(y)=y+{A}_1y^{- \delta }C_\beta^{\delta+1}C_{ \delta+1}$
 for $y>0.$
For any fixed $\delta\geq1,\chi,C_\beta,C_{\delta+1}>0,$
Then $$
\min_{y>0}H(y)=\frac{(\delta-1)}{\delta}C_{\delta+1}^{\frac{1}{\delta+1}}\chi$$
and
\begin{equation}\label{gedrrrrsderrfggddffffnjjmmx1.731426677gg}\min_{y>0}\tilde{H}(y)=\frac{(\delta-1)}{\delta}C_{\delta+1}^{\frac{1}{\delta+1}}C_\beta.
\end{equation}
\end{lemma}
\begin{proof}
It is easy to verify that $$H'(y)=1- A_1\delta C_{\delta+1}\left(\frac{\chi }{y} \right)^{\delta+1}.$$
Let $H'(y)=0$, we have
$$y=\left(A_1C_{\delta+1}\delta\right)^{\frac{1}{\delta+1}}\chi.$$
On the other hand, by $\lim_{y\rightarrow0^+}H(y)=+\infty$ and $\lim_{y\rightarrow+\infty}H(y)=+\infty$, we have
\begin{equation}\label{grttttttddffffnjjmmghuuuiiihhhhx1.731426677gg}
\begin{array}{rl}
\min_{y>0}H(y)=H[\left(A_1C_{\delta+1}\delta\right)^{\frac{1}{\delta+1}}\chi]=&\displaystyle{\left(A_1C_{\delta+1}\right)^{\frac{1}{\delta+1}}
(\delta^{\frac{1}{\delta+1}}+\delta^{-\frac{\delta}{\delta+1}})\chi}\\
=&\displaystyle{\frac{(\delta-1)}{\delta}C_{\delta+1}^{\frac{1}{\delta+1}}\chi.}\\
\end{array}
\end{equation}
Employing  the same arguments as in the proof of \dref{grttttttddffffnjjmmghuuuiiihhhhx1.731426677gg},  we conclude \dref{gedrrrrsderrfggddffffnjjmmx1.731426677gg}.
\end{proof}

\begin{lemma}\label{lemma45630223}
Let $r=2$ and  $(u,v,w)$ be a solution to \dref{1.1} on $(0,T_{max})$.
If
  \begin{equation}\label{gddffffnjjmmx1.731426677gg}
\mu>\frac{(N-2)_{+}}{N}(\chi+C_{\beta}) C^{\frac{1}{\frac{N}{2}+1}}_{\frac{N}{2}+1},
\begin{array}{ll}\\
 \end{array}
\end{equation}
then 
 for all $p>1$,
there exists a positive constant $C:=C(p,|\Omega|,\mu,\chi,\xi,\beta)$ such that 
\begin{equation}
\int_{\Omega}u^p(x,t)dx\leq C ~~~\mbox{for all}~~ t\in(0,T_{max}).
\label{zjscz2.5297x96302222114}
\end{equation}
\end{lemma}
\begin{proof}
Due to  $\mu>\frac{(N-2)_{+}}{N}(\chi+C_{\beta}) C^{\frac{1}{\frac{N}{2}+1}}_{\frac{N}{2}+1}$,  one can choose  ${q_0}>\frac{N}{2}$ 
such that
\begin{equation}\mu>\frac{{q_0}-1}{{q_0}}(C_{\beta}+\chi) C_{{q_0}+1}^{\frac{1}{{q_0}+1}}.
\label{3333ddfffggcz2.5114114}
\end{equation}
Let $l=q_0.$
Multiplying the first equation of \dref{1.1}
  by $u^{{l}-1}$ and integrating over $\Omega$, 
 we get
\begin{equation}
\begin{array}{rl}
&\displaystyle{\frac{1}{{l}}\frac{d}{dt}\|u\|^{{l}}_{L^{{l}}(\Omega)}+({{l}-1})\int_{\Omega}u^{{{l}-2}}|\nabla u|^2dx}
\\
=&\displaystyle{-\chi\int_\Omega \nabla\cdot( u\nabla v)
  u^{{l}-1} dx-\xi\int_\Omega \nabla\cdot( u\nabla w)
  u^{{l}-1} dx+
\int_\Omega   u^{{l}-1}(au-\mu u^2) dx,}\\
\end{array}
\label{cz2.5114114}
\end{equation}
that is,
\begin{equation}
\begin{array}{rl}
&\displaystyle{\frac{1}{{l}}\frac{d}{dt}\|u\|^{{{l}}}_{L^{{l}}(\Omega)}+({{l}-1})\int_{\Omega}u^{{{l}-2}}|\nabla u|^2dx}
\\
\leq&\displaystyle{-\frac{{l}+1}{{l}}\int_{\Omega} u^{l} dx-\chi\int_\Omega \nabla\cdot( u\nabla v)
  u^{{l}-1} dx}\\
 &\displaystyle{-\xi\int_\Omega \nabla\cdot( u\nabla w)
  u^{{l}-1} dx+\int_\Omega \left(\frac{{l}+1}{{l}} u^{l}+  u^{{l}-1}(au-\mu u^2)\right) dx.}\\
\end{array}
\label{cz2.5kk1214114114}
\end{equation}

Hence, by Young inequality, it reads that
\begin{equation}
\begin{array}{rl}
&\displaystyle{\int_\Omega  \left(\frac{{l}+1}{{l}} u^{l}+ u^{{l}-1}(au-\mu u^2)\right) dx}\\
\leq &\displaystyle{\frac{{l}+1}{{l}}\int_\Omega u^{l} dx+a\int_\Omega    u^{{l}}dx- \mu\int_\Omega  u^{{{l}+1}}dx}\\
\leq &\displaystyle{(\varepsilon_1- \mu)\int_\Omega u^{{{l}+1}} dx+C_1(\varepsilon_1,{l}),}
\end{array}
\label{cz2.563011228ddff}
\end{equation}
where $$
\varepsilon_1=\frac{1}{4}(\mu -\frac{{q_0}-1}{{q_0}}(C_{\beta}+\chi)C_{{q_0}+1}^{\frac{1}{{q_0}+1}})>0
$$
and
$$C_1(\varepsilon_1,{l})=\frac{1}{{l}+1}\left(\varepsilon_1\frac{{l}+1}{{l}}\right)^{-{l} }
\left(\frac{{l}+1}{{l}}+a\right)^{{l}+1 }|\Omega|.$$
%
%

Next,
integrating by parts to the first term on the right hand side of \dref{cz2.5114114},
we obtain 
\begin{equation}
\begin{array}{rl}
&\displaystyle{-\chi\int_\Omega \nabla\cdot( u\nabla v)
  u^{{l}-1} dx}
\\
=&\displaystyle{({{l}-1})\chi\int_\Omega  u^{{{l}-1}}\nabla u\cdot\nabla v dx}
\\
\leq&\displaystyle{\frac{{{l}-1}}{{l}}\chi \int_\Omega u^{{l}}|\Delta v| dx.}
\\
\end{array}
\label{cz2.563019114}
\end{equation}
%

Next, due to \dref{vbgncz2.5xx1ffgghh512} and the Young inequality,  we derive that %
there exist positive constant $C_2:=(\frac{1}{2}\frac{1}{l-1}C_\beta^2l^2+C_\beta)$ and $C_3:=\frac{1}{l+1}(\varepsilon_3\frac{l+1}{l})^{-l}C_2^{l+1}$ such that
\begin{equation}
\begin{array}{rl}
&-\xi\displaystyle{\int_\Omega \nabla\cdot( u\nabla w)
  u^{{l}-1} dx}\\
  \leq &\displaystyle{C_\beta(\frac{l-1}{l}
\int_\Omega  u^{l}(v+1)+l\int_{\Omega} u^{l-1}|\nabla u|)}
\\
\leq &\displaystyle{\frac{l-1}{2}\int_{\Omega} u^{l-2}|\nabla u|^2+(\frac{1}{2}\frac{1}{l-1}C_\beta^2l^2+\frac{l-1}{l}C_\beta)\int_\Omega  u^{l}+
C_\beta \frac{l-1}{l}\int_\Omega  u^{l}v}
\\
\leq &\displaystyle{\frac{l-1}{2}\int_{\Omega} u^{l-2}|\nabla u|^2+C_2\int_\Omega  u^{l}+
C_\beta \frac{l-1}{l}\int_\Omega  u^{l}v}
\\
\leq &\displaystyle{\frac{l-1}{2}\int_{\Omega} u^{l-2}|\nabla u|^2+(\varepsilon_2+\varepsilon_3)\int_\Omega  u^{l+1}}\\
&+\displaystyle{\frac{1}{l+1}(\varepsilon_2\frac{l+1}{l})^{-l}C_\beta^{l+1}\left(\frac{l-1}{l}\right)^{l+1}\int_\Omega v^{l+1}+C_3,}\\
\end{array}
\label{czssddee2.5630ssddddd19114}
\end{equation}
where $\varepsilon_2:=\tilde{\lambda}_0$, $
\varepsilon_3=\frac{1}{4}(\mu -\frac{{q_0}-1}{{q_0}}(C_{\beta}+\chi)C_{{q_0}+1}^{\frac{1}{{q_0}+1}})>0
$
and
\begin{equation}
\tilde{\lambda}_0:=\left(A_1C_{{l}+1}{l}\right)^{\frac{1}{{l}+1}}C_\beta.
\label{cz2.56ssd30er191rrrr1ddfrttt4}
\end{equation}
Here  $A_1$ is given by \dref{zjscz2.5297x9630222211444125}.

Now, let \begin{equation}
\lambda_0:=\left(A_1C_{{l}+1}{l}\right)^{\frac{1}{{l}+1}}\chi.
\label{cz2.56ssd30er191rrrr14}
\end{equation}
While from \dref{cz2.563019114} and the Young inequality, we
have
\begin{equation}
\begin{array}{rl}
&\displaystyle{-\chi\int_\Omega \nabla\cdot( u\nabla v)
  u^{{l}-1} dx}
\\
\leq&\displaystyle{\lambda_0\int_\Omega  u^{{l}+1}dx+\frac{1}{ { {l}+1}}\left[\lambda_0\frac{ { {l}+1}}{ {l}}\right]^{- {l} }\left(\frac{{l}-1}{{l}}\chi \right)^{ { {l}+1}}\int_\Omega |\Delta v|^{ { {l}+1}} dx}
\\
=&\displaystyle{\lambda_0\int_\Omega  u^{{l}+1}dx+{A}_1\lambda_0^{- {l} }\chi^{ { {l}+1}}\int_\Omega |\Delta v|^{ { {l}+1}} dx,}
\\
\end{array}
\label{cz2.563019114gghh}
\end{equation}
where
${A}_1$ is given by \dref{zjscz2.5297x9630222211444125}.
Thus, inserting \dref{cz2.563011228ddff}, \dref{czssddee2.5630ssddddd19114} and \dref{cz2.563019114gghh} into \dref{cz2.5kk1214114114}, we get
\begin{equation*}
\begin{array}{rl}
\displaystyle\frac{1}{{l}}\displaystyle\frac{d}{dt}\|u\|^{{{l}}}_{L^{{l}}(\Omega)}+\frac{l-1}{2}\int_{\Omega} u^{l-2}|\nabla u|^2\leq&\displaystyle{(\varepsilon_1+\tilde{\lambda}_0+\varepsilon_3+\lambda_0- \mu)\int_\Omega u^{{{l}+1}} dx-\frac{{l}+1}{{l}}\int_{\Omega} u^{l} dx}\\
&+\displaystyle{{A}_1\lambda_0^{- {l} }\chi^{ {l}+1}\int_\Omega |\Delta v|^{ {l}+1} dx}\\
&+\displaystyle{{A}_1\tilde{\lambda}_0^{- {l} }C_\beta^{ { {l}+1}}\int_\Omega v^{l+1}+
C_1+C_3.}\\
\end{array}
\end{equation*}
For any $t\in (s_0,T_{max})$,
employing the variation-of-constants formula to the above inequality, we obtain
\begin{equation}
\begin{array}{rl}
&\displaystyle{\frac{1}{{l}}\|u(t) \|^{{{l}}}_{L^{{l}}(\Omega)}}
\\
\leq&\displaystyle{\frac{1}{{l}}e^{-( { {l}+1})(t-s_0)}\|u(s_0) \|^{{{l}}}_{L^{{l}}(\Omega)}+(\varepsilon_1+\tilde{\lambda}_0+\varepsilon_3+\lambda_0- \mu)\int_{s_0}^t
e^{-( { {l}+1})(t-s)}\int_\Omega u^{{{l}+1}} dxds}\\
&+\displaystyle{{A}_1\lambda_0^{- {l} }\chi^{ {l}+1}\int_{s_0}^t
e^{-( { {l}+1})(t-s)}\int_\Omega |\Delta v|^{ {l}+1} dxds}\\
&+\displaystyle{(C_1+C_3)\int_{s_0}^t
e^{-( { {l}+1})(t-s)}ds+{A}_1\tilde{\lambda}_0^{- {l} }C_\beta^{ { {l}+1}}\int_{s_0}^t
e^{-( { {l}+1})(t-s)}\int_\Omega v^{{{l}+1}} dxds}\\
\leq&\displaystyle{(\varepsilon_1+\tilde{\lambda}_0+\varepsilon_3+\lambda_0- \mu)\int_{s_0}^t
e^{-( { {l}+1})(t-s)}\int_\Omega u^{{{l}+1}} dxds}\\
&\displaystyle{+{A}_1\lambda_0^{- {l} }\chi^{ {l}+1}\int_{s_0}^t
e^{-( { {l}+1})(t-s)}\int_\Omega |\Delta v|^{ {l}+1} dxds}\\
&+\displaystyle{{A}_1\tilde{\lambda}_0^{- {l} }C_\beta^{ { {l}+1}}\int_{s_0}^t
e^{-( { {l}+1})(t-s)}\int_\Omega v^{{{l}+1}} dxds+C_4,}\\
\end{array}
\label{cz2.5kk1214114114rrgg}
\end{equation}
where
$$C_4:=C_4(\varepsilon_1,\varepsilon_3,{l})=\frac{1}{{l}}\|u(s_0) \|^{{{l}}}_{L^{{l}}(\Omega)}+
 (C_1+C_3)\int_{s_0}^t
e^{-( { {l}+1})(t-s)}ds.$$
Now, 
by \dref{cz2.56ssd30er191rrrr1ddfrttt4},  Lemma \ref{lemma45xy1222232} and the second equation of \dref{1.1}, we have
\begin{equation}\label{cz2.5kke34567789999001214114114rrggjjkk}
\begin{array}{rl}
&\displaystyle{{A}_1\lambda_0^{- {l} }\chi^{ {l}+1}\int_{s_0}^t
e^{-( { {l}+1})(t-s)}\int_\Omega |\Delta v|^{ {l}+1} dxds}
\\
=&\displaystyle{{A}_1\lambda_0^{- {l} }\chi^{ {l}+1}e^{-( { {l}+1})t}\int_{s_0}^t
e^{( { {l}+1})s}\int_\Omega |\Delta v|^{ {l}+1} dxds}\\
\leq&\displaystyle{{A}_1\lambda_0^{- {l} }\chi^{ {l}+1}e^{-( { {l}+1})t}C_{ {l}+1}(\int_{s_0}^t
\int_\Omega e^{( { {l}+1})s}u^{ {l}+1} dxds+e^{( { {l}+1})s_0}\|v(s_0,t)\|^{ {l}+1}_{W^{2, { {l}+1}}})}\\
\end{array}
\end{equation}
and
\begin{equation}\label{cz2.5kk1214114114rrggjjkk}
\begin{array}{rl}
&\displaystyle{{A}_1\tilde{\lambda}_0^{- {l} }C_\beta^{ { {l}+1}}\int_{s_0}^t
e^{-( { {l}+1})(t-s)}\int_\Omega v^{ {l}+1} dxds}
\\
=&\displaystyle{{A}_1\tilde{\lambda}_0^{- {l} }C_\beta^{ { {l}+1}}e^{-( { {l}+1})t}\int_{s_0}^t
e^{( { {l}+1})s}\int_\Omega  v^{ {l}+1} dxds}\\
\leq&\displaystyle{{A}_1\tilde{\lambda}_0^{- {l} }C_\beta^{ { {l}+1}}e^{-( { {l}+1})t}C_{ {l}+1}(\int_{s_0}^t
\int_\Omega e^{( { {l}+1})s}u^{ {l}+1} dxds+e^{( { {l}+1})s_0}\|v(s_0,t)\|^{ {l}+1}_{W^{2, { {l}+1}}})}\\
\end{array}
\end{equation}
for all $t\in(s_0, T_{max})$.
By substituting \dref{cz2.5kke34567789999001214114114rrggjjkk}--\dref{cz2.5kk1214114114rrggjjkk} into \dref{cz2.5kk1214114114rrgg}, using \dref{cz2.56ssd30er191rrrr14} and Lemma \ref{lemma45630223116}, we get
\begin{equation}
\begin{array}{rl}
&\displaystyle{\frac{1}{{l}}\|u(t) \|^{{{l}}}_{L^{{l}}(\Omega)}}
\\
\leq&\displaystyle{(\varepsilon_1+\tilde{\lambda}_0+{A}_1\tilde{\lambda}_0^{- {l} }C_{\beta}^{ {l}+1}C_{ {l}+1}+\lambda_0+{A}_1\lambda_0^{- {l} }\chi^{ {l}+1}C_{ {l}+1}- \mu)\int_{s_0}^t
e^{-( {l}+1)(t-s)}\int_\Omega u^{{{l}+1}} dxds}\\
&+\displaystyle{{A}_1(\lambda_0^{- {l} }\chi^{ {l}+1}+\tilde{\lambda}_0^{- {l} }C_\beta^{ { {l}+1}})e^{-( {l}+1)(t-s_0)}C_{ {l}+1}\|v(s_0,t)\|^{ {l}+1}_{W^{2, { {l}+1}}}+C_4}\\
=&\displaystyle{(\varepsilon_1+\varepsilon_3+\frac{({l}-1)}{{l}}C_{{l}+1}^{\frac{1}{{l}+1}}C_{\beta}+\frac{({l}-1)}{{l}}C_{{l}+1}^{\frac{1}{{l}+1}}\chi- \mu)\int_{s_0}^t
e^{-( { {l}+1})(t-s)}\int_\Omega u^{{{l}+1}} dxds}\\
&+\displaystyle{{A}_1(\lambda_0^{- {l} }\chi^{ {l}+1}+\tilde{\lambda}_0^{- {l} }C_\beta^{ { {l}+1}})e^{-( { {l}+1})(t-s_0)}C_{ {l}+1}\|v(s_0,t)\|^{ {l}+1}_{W^{2, { {l}+1}}}+C_4.}\\
\end{array}
\label{cz2.5kk1214114114rrggkkll}
\end{equation}
Since $l=q_0$, therefore, $$\frac{({l}-1)}{{l}}C_{{l}+1}^{\frac{1}{{l}+1}}C_{\beta}+\frac{({l}-1)}{{l}}C_{{l}+1}^{\frac{1}{{l}+1}}\chi- \mu=\frac{{q_0}-1}{{q_0}}(C_{\beta}+\chi) C_{{q_0}+1}^{\frac{1}{{q_0}+1}}- \mu,$$
so that,
 \begin{equation}0<\varepsilon_1+\varepsilon_3=\frac{1}{2}(\mu -\frac{{q_0}-1}{{q_0}}(C_{\beta}+\chi) C_{{q_0}+1}^{\frac{1}{{q_0}+1}})<\mu -\frac{{q_0}-1}{{q_0}}(C_{\beta}+\chi) C_{{q_0}+1}^{\frac{1}{{q_0}+1}}.
\label{cz2.5kk1214114114ssdddrrttrrgssdeersdddtttgkkll}
\end{equation}
Collecting \dref{cz2.5kk1214114114ssdddrrttrrgssdeersdddtttgkkll} and  \dref{cz2.5kk1214114114rrggkkll}, we derive that there exists a positive constant $C_5$
such that
\begin{equation}
\begin{array}{rl}
&\displaystyle{\int_{\Omega}u^{{q_0}}(x,t) dx\leq C_5~~\mbox{for all}~~t\in (s_0, T_{max}).}\\
\end{array}
\label{cz2.5kk1214114114rrggkklljjuu}
\end{equation}
Next,
we fix $q <\frac{N{q_0}}{(N-{q_0})^+}$
and choose some
 $\alpha> \frac{1}{2}$ such that
\begin{equation}
q <\frac{1}{\frac{1}{q_0}-\frac{1}{N}+\frac{2}{N}(\alpha-\frac{1}{2})}\leq\frac{N{q_0}}{(N-{q_0})^+}.
\label{fghgbhnjcz2.5ghju48cfg924ghyuji}
\end{equation}

Now, involving the variation-of-constants formula
for $v$, we have
\begin{equation}
v(t)=e^{-t(A+1)}v(s_0) +\int_{s_0}^{t}e^{-(t-s)(A+1)}u(s) ds,~~ t\in(s_0, T_{max}).
\label{fghbnmcz2.5ghju48cfg924ghyuji}
\end{equation}
Hence, it follows from \dref{eqx45xx12112} and  
 \dref{fghbnmcz2.5ghju48cfg924ghyuji} that
\begin{equation}
\begin{array}{rl}
&\displaystyle{\|(A+1)^\alpha v(t)\|_{L^q(\Omega)}}\\
\leq&\displaystyle{C_6\int_{s_0}^{t}(t-s)^{-\alpha-\frac{N}{2}(\frac{1}{q_0}-\frac{1}{q})}e^{-\mu(t-s)}\|u(s)\|_{L^{q_0}(\Omega)}ds+
C_6s_0^{-\alpha-\frac{N}{2}(1-\frac{1}{q})}\|v(s_0,t)\|_{L^1(\Omega)}}\\
\leq&\displaystyle{C_6\int_{0}^{+\infty}\sigma^{-\alpha-\frac{N}{2}(\frac{1}{q_0}-\frac{1}{q})}e^{-\mu\sigma}d\sigma
+C_6s_0^{-\alpha-\frac{N}{2}(1-\frac{1}{q})}\beta.}\\
\end{array}
\label{gnhmkfghbnmcz2.5ghju48cfg924ghyuji}
\end{equation}
Hence, due to \dref{fghgbhnjcz2.5ghju48cfg924ghyuji}  and \dref{gnhmkfghbnmcz2.5ghju48cfg924ghyuji}, we have
\begin{equation}
\int_{\Omega}|\nabla {v}(t)|^{q}\leq C_7~~\mbox{for all}~~ t\in(s_0, T_{max})
\label{ffgbbcz2.5ghju48cfg924ghyuji}
\end{equation}
and $q\in[1,\frac{N{q_0}}{(N-{q_0})^+})$.
Finally, in view of \dref{eqx45xx12112} and \dref{ffgbbcz2.5ghju48cfg924ghyuji},
 we can get \begin{equation}
\int_{\Omega}|\nabla {v}(t)|^{q}\leq C_8~~\mbox{for all}~~ t\in(0, T_{max})~~\mbox{and}~~q\in[1,\frac{N{q_0}}{(N-{q_0})^+})
\label{ffgbbcz2.5ghjusseeeddd48cfg924ghyuji}
\end{equation}
with some positive constant $C_8.$

Multiplying both sides of the first equation in \dref{1.1} by $u^{p-1}$, integrating over $\Omega$ and integrating by parts, we arrive at
\begin{equation}
\begin{array}{rl}
&\displaystyle{\frac{1}{{p}}\frac{d}{dt}\|u\|^{{p}}_{L^{{p}}(\Omega)}+({{p}-1})\int_{\Omega}u^{{{p}-2}}|\nabla u|^2dx}
\\
=&\displaystyle{-\chi\int_\Omega \nabla\cdot( u\nabla v)
  u^{{p}-1} dx-\xi\int_\Omega \nabla\cdot( u\nabla w)
  u^{{p}-1} dx+
\int_\Omega   u^{{p}-1}(au-\mu u^2) dx}\\
=&\displaystyle{\chi({p}-1)\int_\Omega  u^{{p}-1}\nabla u\cdot\nabla v
   dx+\xi({p}-1)\int_\Omega  u^{{p}-1}\nabla u\cdot\nabla w
   dx+
\int_\Omega   u^{{p}-1}(au-\mu u^2) dx,}\\
\end{array}
\label{cz2aasweee.5114114}
\end{equation}
which together with the Young inequality and  \dref{vbgncz2.5xx1ffgghh512} implies that
\begin{equation}
\begin{array}{rl}
&\displaystyle{\frac{1}{{p}}\frac{d}{dt}\|u\|^{{p}}_{L^{{p}}(\Omega)}+({{p}-1})\int_{\Omega}u^{{{p}-2}}|\nabla u|^2dx}
\\
\leq&\displaystyle{\frac{{{p}-1}}{2}\int_{\Omega}u^{{{p}-2}}|\nabla u|^2dx+
\frac{\chi^2({p}-1)}{2}\int_\Omega  u^{{p}}|\nabla v|^2
   dx+C_9\int_\Omega  v^{p+1}-\frac{\mu }{2}\int_\Omega u^{p+1}dx+ C_{10}}\\
\end{array}
\label{cz2aasweee.5ssedfff114114}
\end{equation}
for some positive constants $C_9$ and $C_{10}.$
Now, in light of $q_0>\frac{N}{2}$, due to \dref{ffgbbcz2.5ghjusseeeddd48cfg924ghyuji} and the Sobolev imbedding theorem, we derive that
there exists a positive constant $C_{11}$ such that
\begin{equation}
\label{cz2aasweee.5ssedfff114114}
C_9\int_\Omega  v^{p+1}(x,t)\leq C_{11}~~\mbox{for all}~~ t\in(0, T_{max})~~\mbox{and}~~p>1.
\end{equation}
Since, $q_0>\frac{N}{2}$ yields $q_0<\frac{N{q_0}}{2(N-{q_0})^+}$, in light of the H\"{o}lder inequality, \dref{eqx45xx12112} and \dref{ffgbbcz2.5ghjusseeeddd48cfg924ghyuji}, we arrive   at
\begin{equation}
\begin{array}{rl}
 \displaystyle\frac{\chi^2({p}-1)}{2}\displaystyle\int_\Omega{{u^{p }}} |\nabla {v}|^2\leq&\displaystyle{ \displaystyle\frac{\chi^2({p}-1)}{2}\left(\displaystyle\int_\Omega{{u^{\frac{q_0}{q_0-1} p }}}\right)^{\frac{q_0-1}{q_0}}\left(\displaystyle\int_\Omega |\nabla {v}|^{2q_0}\right)^{\frac{1}{q_0}}}\\
\leq&\displaystyle{C_{12}\|  {{u^{\frac{p}{2}}}}\|^{2}_{L^{2\frac{q_0}{q_0-1} }(\Omega)},}\\
\end{array}
\label{cz2.57151hhkkhhhjukildrfthjjhhhhh}
\end{equation}
where $C_{10}$ is a positive constant.
Since 
$q_0> \frac{N}{2}$ and $p>q_0-1$,
we have
$$\frac{q_0}{p}\leq\frac{q_0}{q_0-1}\leq\frac{N}{N-2},$$
which together with the Gagliardo--Nirenberg inequality (see e.g.  \cite{Zhejjjjnssdssdddfffgghjjkk1})
 implies that
\begin{equation}
\begin{array}{rl}
C_{12}\|  {{u^{\frac{p}{2}}}}\|
^{2}_{L^{2\frac{q_0}{q_0-1} }(\Omega)}\leq&\displaystyle{C_{13}(\|\nabla   {{u^{\frac{p}{2}}}}\|_{L^2(\Omega)}^{\mu_1}\|  {{u^{\frac{p}{2}}}}\|_{L^\frac{2q_0}{p}(\Omega)}^{1-\mu_1}+\|  {{u^{\frac{p}{2}}}}\|_{L^\frac{2q_0}{p}(\Omega)})^{2}}\\
\leq&\displaystyle{C_{14}(\|\nabla   {{u^{\frac{p}{2}}}}\|_{L^2(\Omega)}^{2\mu_1}+1)}\\
=&\displaystyle{C_{14}(\|\nabla   {{u^{\frac{p}{2}}}}\|_{L^2(\Omega)}^{\frac{2N(p-q_0+1)}{Np+2q_0-Nq_0}}+1)}\\
\end{array}
\label{cz2.563022222ikopl2sdfg44}
\end{equation}
with some positive constants $C_{13}, C_{14}$ and
$$\mu_1=\frac{\frac{N{p}}{2q_0}-\frac{Np}{2\frac{q_0}{q_0-1} p }}{1-\frac{N}{2}+\frac{N{p}}{2q_0}}=
{p}\frac{\frac{N}{2q_0}-\frac{N}{2\frac{q_0}{q_0-1} p }}{1-\frac{N}{2}+\frac{N{p}}{2q_0}}\in(0,1).$$
Now, in view of the Young inequality, we derive that
\begin{equation}
\begin{array}{rl}
\displaystyle\frac{\chi^2({p}-1)}{2}\displaystyle\int_\Omega  u^{{p}}|\nabla v|^2dx &\leq\displaystyle{\frac{{{p}-1}}{4}\int_{\Omega}u^{{{p}-2}}|\nabla u|^2dx+C_{15}.}\\
\end{array}
\label{cz2aasweeeddfff.5ssedfssddff114114}
\end{equation}
Inserting \dref{cz2aasweeeddfff.5ssedfssddff114114} into \dref{cz2aasweee.5ssedfff114114}, we conclude that
\begin{equation}
\begin{array}{rl}
&\displaystyle{\frac{1}{{p}}\frac{d}{dt}\|u\|^{{p}}_{L^{{p}}(\Omega)}+\frac{{{p}-1}}{4}\int_{\Omega}u^{{{p}-2}}|\nabla u|^2dx+
\frac{\mu}{2}\int_\Omega u^{p+1}dx\leq C_{16}.}\\
\end{array}
\label{cz2aasweee.5ssedfff114114}
\end{equation}
Therefore, integrating the above inequality  with respect to $t$ yields
\begin{equation}
\begin{array}{rl}
\|u(\cdot, t)\|_{L^{{p}}(\Omega)}\leq C_{17} ~~ \mbox{for all}~~p\geq1~~\mbox{and}~~  t\in(0,T_{max}) \\
\end{array}
\label{cz2.5g556789hhjui78jj90099}
\end{equation}
for some positive constant $C_{17}$.
The proof Lemma \ref{lemma45630223} is complete.
\end{proof}

\begin{lemma}\label{99lemma45630223}
Let $(u,v,w)$ be a solution to \dref{1.1} on $(0,T_{max})$. Assume that  $r>2$.
 Then 
 for all $p>1$,
there exists a positive constant $C:=C(p,|\Omega|,r,\mu,\xi,\chi,\beta)$ such that 
\begin{equation}
\int_{\Omega}u^p(x,t)dx\leq C ~~~\mbox{for all}~~ t\in(0,T_{max}).
\label{99zjscz2.5297x96302222114}
\end{equation}
\end{lemma}
\begin{proof}
Firstly, multiplying the first equation of \dref{1.1}
  by $u^{{l}-1}$ and integrating over $\Omega$, 
 we get
\begin{equation}
\begin{array}{rl}
&\displaystyle{\frac{1}{{l}}\frac{d}{dt}\|u\|^{{l}}_{L^{{l}}(\Omega)}+({{l}-1})\int_{\Omega}u^{{{l}-2}}|\nabla u|^2dx}
\\
=&\displaystyle{-\chi\int_\Omega \nabla\cdot( u\nabla v)
  u^{{l}-1} dx-\xi\int_\Omega \nabla\cdot( u\nabla w)
  u^{{l}-1} dx+
\int_\Omega   u^{{l}-1}(au-\mu u^r) dx.}\\
\end{array}
\label{9999999cz2.5114114}
\end{equation}
In light of the  Young inequality and $r>2$, it reads that there exists a positive constant $C_1$ such that
\begin{equation}
\begin{array}{rl}
&\displaystyle{\int_\Omega  \left(\frac{{l}+1}{{l}} u^{l}+ u^{{l}-1}(au-\mu u^r)\right) dx}\\
\leq &\displaystyle{\frac{{l}+1}{{l}}\int_\Omega u^{l} dx+a\int_\Omega    u^{{l}}dx- \mu\int_\Omega  u^{{{l+r-1}}}dx}\\
\leq &\displaystyle{- \frac{7\mu}{8}\int_\Omega u^{{{l+r-1}}} dx+C_1.}
\end{array}
\label{999999cz2.563011228ddff}
\end{equation}
%
%

Next,
integrating by parts to the first term on the right hand side of \dref{9999999cz2.5114114} and  using  the Young inequality,
we obtain 
\begin{equation}
\begin{array}{rl}
&\displaystyle{-\chi\int_\Omega \nabla\cdot( u\nabla v)
  u^{{l}-1} dx}
\\
\leq&\displaystyle{\frac{{{l}-1}}{{l}}\chi \int_\Omega u^{{l}}|\Delta v| dx}
\\
\leq&\displaystyle{\frac{\mu}{8}\int_\Omega  u^{{l+r-1}}dx+C_2\int_\Omega |\Delta v|^{\frac{r+l-1}{r-1}} dx}
\\
\leq&\displaystyle{\frac{\mu}{8}\int_\Omega  u^{{l+r-1}}dx+\int_\Omega |\Delta v|^{l+1} dx+C_3.}
\\
\end{array}
\label{eert99cz2.563019114}
\end{equation}

Next, due to \dref{vbgncz2.5xx1ffgghh512} and the Young inequality,  we derive that there exist positive constants $C_4,C_5$ and
 $C_6$ such that
\begin{equation}
\begin{array}{rl}
-\xi\displaystyle\int_\Omega \nabla\cdot( u\nabla w)
  u^{{l}-1} dx\leq &\displaystyle{C_\beta(\frac{l-1}{l}
\int_\Omega  u^{l}(v+1)+l\int_{\Omega} u^{l-1}|\nabla u|)}
\\
\leq &\displaystyle{C_4(
\int_\Omega  u^{l}(v+1)+l\int_{\Omega} u^{l-1}|\nabla u|)}
\\
\leq &\displaystyle{\frac{l-1}{2}\int_{\Omega} u^{l-2}|\nabla u|^2+C_5\int_\Omega  u^{l}+C_5
\int_\Omega  u^{l}v}
\\
\leq &\displaystyle{\frac{l-1}{2}\int_{\Omega} u^{l-2}|\nabla u|^2+\frac{\mu}{8}\int_\Omega  u^{l+r-1}+
\int_\Omega v^{l+1}+C_6.}
\\
\end{array}
\label{eertt56czssddee2.5630ssddddd19114}
\end{equation}
Thus, inserting \dref{999999cz2.563011228ddff}--\dref{cz2.563019114gghh} into \dref{9999999cz2.5114114}, we get
\begin{equation}\label{eertt56czssdderrssddee2.5630ssddddd19114}
\begin{array}{rl}
\displaystyle\frac{1}{{l}}\displaystyle\frac{d}{dt}\|u\|^{{{l}}}_{L^{{l}}(\Omega)}+\frac{l-1}{2}\int_{\Omega} u^{l-2}|\nabla u|^2\leq&\displaystyle{- \frac{5\mu}{8}\int_\Omega u^{{{l+r-1}}} dx-\frac{{l}+1}{{l}}\int_{\Omega} u^{l} dx}\\
&+\displaystyle{\int_\Omega |\Delta v|^{ {l}+1} dx+\int_\Omega v^{l+1}+
C_7.}\\
\end{array}
\end{equation}

For any $t\in (s_0,T_{max})$,
applying the variation-of-constants formula to \dref{eertt56czssdderrssddee2.5630ssddddd19114}, we get
\begin{equation}
\begin{array}{rl}
&\displaystyle{\frac{1}{{l}}\|u(t) \|^{{{l}}}_{L^{{l}}(\Omega)}}
\\
\leq&\displaystyle{\frac{1}{{l}}e^{-( { {l}+1})(t-s_0)}\|u(s_0) \|^{{{l}}}_{L^{{l}}(\Omega)}- \frac{5\mu}{8}\int_{s_0}^t
e^{-( { {l}+1})(t-s)}\int_\Omega u^{{{l+r}-1}} dxds}\\
&+\displaystyle{\int_{s_0}^t
e^{-( { {l}+1})(t-s)}\int_\Omega |\Delta v|^{ {l}+1} dxds+ C_7\int_{s_0}^t
e^{-( { {l}+1})(t-s)}ds+\int_{s_0}^t
e^{-( { {l}+1})(t-s)}\int_\Omega v^{{{l}+1}} dxds}\\
\leq&\displaystyle{- \frac{5\mu}{8}\int_{s_0}^t
e^{-( { {l}+1})(t-s)}\int_\Omega u^{{{l+r}-1}} dxds+\int_{s_0}^t
e^{-( { {l}+1})(t-s)}\int_\Omega (|\Delta v|^{ {l}+1}+v^{{{l}+1}}) dxds+C_8,}\\
\end{array}
\label{1122cz2.5kk1214114114rrgg}
\end{equation}
where
$$C_8:=\frac{1}{{l}}\|u(s_0) \|^{{{l}}}_{L^{{l}}(\Omega)}+
 C_7\int_{s_0}^t
e^{-( { {l}+1})(t-s)}ds.$$
Now, by Lemma \ref{lemma45xy1222232}, we have
\begin{equation}\label{1122cz2.5kk1214114114rrggjjkk}
\begin{array}{rl}
&\displaystyle{\int_{s_0}^t
e^{-( { {l}+1})(t-s)}\int_\Omega (|\Delta v|^{ {l}+1}+|v|^{ {l}+1} ) dxds}
\\
=&\displaystyle{e^{-( { {l}+1})t}\int_{s_0}^t
e^{( { {l}+1})s}\int_\Omega  (|\Delta v|^{ {l}+1}+|v|^{ {l}+1} ) dxds}\\
\leq&\displaystyle{e^{-( { {l}+1})t}C_{ {l}+1}(\int_{s_0}^t
\int_\Omega e^{( { {l}+1})s}u^{ {l}+1} dxds+e^{( { {l}+1})s_0}\|v(s_0,t)\|^{ {l}+1}_{W^{2, { {l}+1}}})}\\
\end{array}
\end{equation}
for all $t\in(s_0, T_{max})$.
By substituting \dref{1122cz2.5kk1214114114rrggjjkk} into \dref{1122cz2.5kk1214114114rrgg} and using the Young inequality, we get
\begin{equation}
\begin{array}{rl}
&\displaystyle{\frac{1}{{l}}\|u(t) \|^{{{l}}}_{L^{{l}}(\Omega)}}
\\
\leq&\displaystyle{- \frac{5\mu}{8}\int_{s_0}^t
e^{-( { {l}+1})(t-s)}\int_\Omega u^{{{l+r}-1}} dxds+C_{ {l}+1}\int_{s_0}^t
e^{-( {l}+1)(t-s)}\int_\Omega u^{{{l}+1}} dxds}\\
&+\displaystyle{e^{-( {l}+1)(t-s_0)}C_{ {l}+1}\|v(s_0,t)\|^{ {l}+1}_{W^{2, { {l}+1}}}+C_8}\\
\leq&\displaystyle{- \frac{\mu}{2}\int_{s_0}^t
e^{-( { {l}+1})(t-s)}\int_\Omega u^{{{l+r}-1}} dxds+C_9}\\
\end{array}
\label{112233cz2.5kk1214114114rrggkkll}
\end{equation}
with
$$\begin{array}{rl}
C_9=&\displaystyle{e^{-( {l}+1)(t-s_0)}C_{ {l}+1}\|v(s_0,t)\|^{ {l}+1}_{W^{2, { {l}+1}}}}\\
&\displaystyle{+\frac{r-2}{r+l-1}\left(\frac{\mu}{8}\frac{r+l-1}{l+1}\right)^{-\frac{l+1}{r-2}}C_{l+1}^{\frac{r+l-1}{r-2}}\frac{\mu}{8}|\Omega|\int_{s_0}^t
e^{-( { {l}+1})(t-s)}ds+C_8.}\\
\end{array}$$
Therefore, integrating \dref{112233cz2.5kk1214114114rrggkkll} respect to $t$ and using \dref{eqx45xx12112} yields
\begin{equation}
\begin{array}{rl}
\|u(\cdot, t)\|_{L^{{l}}(\Omega)}\leq C_{11} ~~ \mbox{for all}~~l\geq1~~\mbox{and}~~  t\in(0,T_{max}) \\
\end{array}
\label{cz2.5g556789hhjui78jj90099}
\end{equation}
for some positive constant $C_{11}$.
The proof Lemma \ref{lemma45630223} is complete.
\end{proof}
Our main result on global existence and boundedness thereby becomes a straightforward consequence
of Lemmata \ref{lemma45630223}--\ref{99lemma45630223} and Lemma \ref{lemma70}.
Indeed, collecting the above Lemmata, in the following,
by invoking a Moser-type iteration (see Lemma A.1 in \cite{Tao794}) and the standard estimate for Neumann semigroup (or the standard parabolic regularity
arguments), we will prove Theorem \ref{theorem3}.

{\bf The proof of Theorem \ref{theorem3}}~
Firstly, due to Lemmata \ref{lemma45630223}--\ref{99lemma45630223}, we derive that there exist positive constants $q_0>N$ and $C_1$ such that
\begin{equation}
\begin{array}{rl}
\|u(\cdot, t)\|_{L^{{q_0}}(\Omega)}\leq C_{1} ~~ \mbox{for all}~~t\in(0,T_{max}). \\
\end{array}
\label{cz2.5g556789hhjssdddui78jj90099}
\end{equation}
Next,
employing  the standard estimate for Neumann semigroup provides $C_2$ and $C_3 > 0$ such that
\begin{equation}
\begin{array}{rl}
&\displaystyle{\|\nabla v(t)\|_{L^\infty(\Omega)}}\\
\leq&\displaystyle{C_2\int_{s_0}^{t}(t-s)^{-\alpha-\frac{N}{2q_0}}e^{-\mu(t-s)}\|u(s)\|_{L^{q_0}(\Omega)}ds+
C_2s_0^{-\alpha}\|v(s_0,t)\|_{L^\infty(\Omega)}}\\
\leq&\displaystyle{C_2\int_{0}^{+\infty}\sigma^{-\alpha-\frac{N}{2q_0}}e^{-\mu\sigma}d\sigma
+C_2s_0^{-\alpha}\beta}\\
\leq&\displaystyle{C_3~~\mbox{for all}~~ t\in(0, T_{max}).}\\
\end{array}
\label{11111gnhmkfghbnmcz2.5ghju48cfg924ghyujiffggg}
\end{equation}
Multiplying both sides of the first equation in \dref{1.1} by $u^{p-1}$, integrating over $\Omega$ and integrating by parts, we conclude that
\begin{equation}
\begin{array}{rl}
&\displaystyle{\frac{1}{{p}}\frac{d}{dt}\|u\|^{{p}}_{L^{{p}}(\Omega)}+({{p}-1})\int_{\Omega}u^{{{p}-2}}|\nabla u|^2dx}
\\
=&\displaystyle{-\chi\int_\Omega \nabla\cdot( u\nabla v)
  u^{{p}-1} dx-\xi\int_\Omega \nabla\cdot( u\nabla w)
  u^{{p}-1} dx+
\int_\Omega   u^{{p}-1}(au-\mu u^r) dx}\\
=&\displaystyle{\chi({p}-1)\int_\Omega  u^{{p}-1}\nabla u\cdot\nabla v
   dx-\xi\int_\Omega \nabla\cdot( u\nabla w)
  u^{{p}-1} dx+
\int_\Omega   u^{{p}-1}(au-\mu u^r) dx.}\\
\end{array}
\label{cz2aasweee.5ssderfff114114}
\end{equation}
Due to  \dref{vbgncz2.5xx1ffgghh512} and \dref{11111gnhmkfghbnmcz2.5ghju48cfg924ghyujiffggg} and the Young inequality, we derive that there exist positive constants $C_4,C_5,C_6$  and  $C_7$   independent of $p$ such that
\begin{equation}
\begin{array}{rl}
\chi({p}-1)\displaystyle\int_\Omega  u^{{p}-1}\nabla u\cdot\nabla v
   dx\leq &\displaystyle{\chi({p}-1)C_4\displaystyle\int_\Omega  u^{{p}-1}|\nabla u|
   dx}\\
\leq &\displaystyle{\frac{p-1}{4}\int_{\Omega} u^{p-2}|\nabla u|^2+C_5p
\int_\Omega  u^{p}}\\
\end{array}
\label{vddfffbgncddfffz2.5xx1ffgghh512}
\end{equation}
and
\begin{equation}
\begin{array}{rl}
-\xi\displaystyle\int_\Omega  u^{p-1}\nabla\cdot(u\nabla w )\leq &\displaystyle{C_6(
\int_\Omega  u^{p}(v+1)+p\int_{\Omega} u^{p-1}|\nabla u|)}\\
\leq &\displaystyle{\frac{p-1}{4}\int_{\Omega} u^{p-2}|\nabla u|^2+C_7p
\int_\Omega  u^{p}.}\\
\end{array}
\label{vddfffbgncddfffz2.5xx1fddffrttfgghh512}
\end{equation}

Hence by \dref{cz2aasweee.5ssderfff114114}--\dref{vddfffbgncddfffz2.5xx1fddffrttfgghh512}, we conclude that there exist positive constants $C_8$ and $C_9$ independent of $p$ such that
\begin{equation}
\begin{array}{rl}
&\displaystyle{\frac{d}{dt}\|u\|^{{p}}_{L^{{p}}(\Omega)}+C_8\int_{\Omega}|\nabla u^{\frac{p}{2}}|^2dx+\int_{\Omega}u^{p}\leq C_9p^2
\int_\Omega  u^{p}.}\\
\end{array}
\label{cz2aasweee.5ssdessderrrfff114114}
\end{equation}
Here and throughout the proof of Theorem \ref{theorem3}, we shall
denote by $C_i(i\in \mathbb{N})$ several positive constants independent of $p$.
Next, with the help of the Gagliardo--Nirenberg inequality, we derive that
\begin{equation}
\begin{array}{rl}
C_9p^2
\displaystyle\int_\Omega  u^{p}=&\displaystyle{ C_9p^2\|u^{\frac{p}{2}}\|_{L^2(\Omega)}^2 }\\
\leq&\displaystyle{ C_9p^2(\|\nabla u^{\frac{p}{2}}\|_{L^{2}(\Omega)}^{2\varsigma_1}
\| u^{\frac{p}{2}}\|_{L^1(\Omega)}^{2(1-\varsigma_1)}+ \| u^{\frac{p}{2}}\|_{L^1(\Omega)}^{2}) }\\
=&\displaystyle{ C_9p^2(\|\nabla u^{\frac{p}{2}}\|_{L^{2}(\Omega)}^{\frac{2N}{N+2}}
\| u^{\frac{p}{2}}\|_{L^1(\Omega)}^{\frac{4}{N+2}}+ \| u^{\frac{p}{2}}\|_{L^1(\Omega)}^{2})}\\
\leq&\displaystyle{  C_9p^2(\|\nabla u^{\frac{p}{2}}\|_{L^{2}(\Omega)}^{\frac{2N}{N+2}}
\| u^{\frac{p}{2}}\|_{L^1(\Omega)}^{\frac{4}{N+2}}+ \| u^{\frac{p}{2}}\|_{L^1(\Omega)}^{2})}\\
\leq&\displaystyle{  C_8\|\nabla u^{\frac{p}{2}}\|_{L^{2}(\Omega)}^{2}+C_{10}p^{N+2}
\| u^{\frac{p}{2}}\|_{L^1(\Omega)}^{2}+  C_9p^2\| u^{\frac{p}{2}}\|_{L^1(\Omega)}^{2}}\\
\leq&\displaystyle{  C_8\|\nabla u^{\frac{p}{2}}\|_{L^{2}(\Omega)}^{2}+C_{11}p^{N+2}
\| u^{\frac{p}{2}}\|_{L^1(\Omega)}^{2},}\\
\end{array}
\label{cz2aasweee.5ssdessderrrffssdffff114114}
\end{equation}
where
$$0<\varsigma_1=\frac{N-\frac{N}{2}}{1-\frac{N}{2}+N}=\frac{N}{N+2}<1,$$
$C_{10}$ and $C_{11}$ are positive constants independent of $p$.
Therefore, inserting \dref{cz2aasweee.5ssdessderrrffssdffff114114} into \dref{cz2aasweee.5ssdessderrrfff114114}, we derive that
\begin{equation}
\begin{array}{rl}
\displaystyle\frac{d}{dt}\| u \|^{p}_{L^p(\Omega)}+\displaystyle\int_\Omega  u ^{p}{}\leq&\displaystyle{ C_{11}{p}^{2+N}
\| u^{\frac{p}{2}}\|_{L^1(\Omega)}^{2}}\\
\leq&\displaystyle{ C_{11}{p}^{2+N}\left(\max\{1,
\| u^{\frac{p}{2}}\|_{L^1(\Omega)}^{2}\right)^2.}\\
\end{array}
\label{zjscz2.5297x9630111rrd67ddfff512}
\end{equation}
Now, choosing $p_i=2^i$ and letting $M_i =\max\{1,\sup_{t\in(0,T)}\int_{\Omega} u ^{\frac{{p_i}}{2}}\}$ for $T\in (0, T_{max})$
and $i=1, 2, \cdots$.
Then \dref{zjscz2.5297x9630111rrd67ddfff512} implies that
\begin{equation}
\begin{array}{rl}
&\displaystyle{\frac{d}{dt}\| u \|^{p_i}_{L^{p_i}(\Omega)}+\int_\Omega  u ^{{p_i}}{}\leq C_{11}{p_i^{2+N}}
M^2_{i-1}(T),}\\
\end{array}
\label{zjscz2.5297x9630111rrd67ddfff512df515}
\end{equation}
which, together with the comparison argument entails that there exists a $\lambda>1$ independent of $i$ such that
\begin{equation}
\begin{array}{rl}
&\displaystyle{M_{i}(T)\leq \max\{\lambda^iM^2_{i-1}(T),|\Omega|\|u_0\|_{L^\infty(\Omega)}^{p_i}\}.}\\
\end{array}
\label{zjscz2.5297x9630111rrd6ssdd7ddfff512df515}
\end{equation}
Now, if $\lambda^iM^2_{i-1}(T)\leq|\Omega|\|u_0\|_{L^\infty(\Omega)}^{p_i}$ for infinitely many
$i\geq 1$, we get
\begin{equation}\|u(\cdot,t)\|_{L^\infty(\Omega)}\leq C_{12}~~~\mbox{for all}~~~t\in(0,T)
\label{zjscz2.5297x963011sdertt1rrd6ssdd7ddfff512df515}
\end{equation}
 with $C_{12}= \|u_0\|_{L^\infty(\Omega)}$.
Otherwise, if $\lambda^iM^2_{i-1}(T)>|\Omega|\|u_0\|_{L^\infty(\Omega)}^{p_i}$ for all sufficiently large $i$, then by \dref{zjscz2.5297x9630111rrd6ssdd7ddfff512df515}, we derive that
\begin{equation}
\begin{array}{rl}
&\displaystyle{M_{i}(T)\leq \lambda^iM^2_{i-1}(T)~~~\mbox{for all sufficiently large}~~~i.}\\
\end{array}
\label{zjscz2.5297x9630111rrd6ssdd7ddssddfffffff512df515}
\end{equation}
 Hence, we may choose $\lambda$ large enough such that
\begin{equation}
\begin{array}{rl}
&\displaystyle{M_{i}(T)\leq \lambda^iM^2_{i-1}(T)~~~\mbox{for all}~~~i\geq1.}\\
\end{array}
\label{zjscz2.5297x9630111rrd6ssdfffffdd7ddssddsddfffffffff512df515}
\end{equation}
 Therefore, in light of
 a straightforward induction (see e.g. Lemma 3.12 of \cite{Tao79477}) we have
\begin{equation}
\begin{array}{rl}
\disp M_{i}(T)\leq&\displaystyle{
\lambda^i
(\lambda^{i-1}M_{i-2}^{2})^{2}}\\
=&\displaystyle{\lambda^{i+2(i-1)}M_{i-2}^{2^2}}\\
\leq&\displaystyle{\lambda^{i+\Sigma_{j=2}^i(j-1)}M_{0}^{2^i}.}\\
\end{array}
\label{cz2.56303hhyy890678789ty4tt8890013378}
\end{equation}
 Taking $p_i$-th
roots on both sides of \dref{cz2.56303hhyy890678789ty4tt8890013378}, with some basic calculation and by taking $T\nearrow T_{max}$,  we can finally conclude that
\begin{equation}\|u(\cdot,t)\|_{L^\infty(\Omega)}\leq C_{13}~~~\mbox{for all}~~~t\in(0,T_{max}).
\label{zjscz2.5297x96ssddd3011sdertt1rrd6ssdd7ddffsdddddf512df515}
\end{equation}
Now, with the above estimate in hand, using \dref{vbgncz2.5xx1ffsskkopffggpgghh512}, we may 
establish
\begin{equation}
\|\nabla w(\cdot,t)\|_{L^{\infty}(\Omega)}  \leq C_{14} ~~\mbox{for all}~~ t\in(0,T_{max}).
\label{hjjkkzjscz2.5297x9630111kkhhii}
\end{equation}
Finally, according to Lemma \ref{lemma70}, this together with \dref{11111gnhmkfghbnmcz2.5ghju48cfg924ghyujiffggg} and \dref{zjscz2.5297x96ssddd3011sdertt1rrd6ssdd7ddffsdddddf512df515}  entails that $(u, v,w)$ is
global in time, and that $u$ is bounded in $\Omega\times(0,\infty)$. $\qed$

{\bf Acknowledgement}:
This work is partially supported by  the National Natural
Science Foundation of China (No. 11601215), Shandong Provincial
Science Foundation for Outstanding Youth (No. ZR2018JL005), Shandong Provincial
Natural Science Foundation,  China (No. ZR2016AQ17) and the Doctor Start-up Funding of Ludong University (No. LA2016006).

\end{document}